\documentclass[10pt]{amsart}
\usepackage{amsmath,amssymb,latexsym,soul}
\usepackage{cite}
\usepackage{color,enumitem,graphicx}
\usepackage[colorlinks=true,urlcolor=blue,
citecolor=red,linkcolor=blue,linktocpage,pdfpagelabels,
bookmarksnumbered,bookmarksopen]{hyperref}
\usepackage[english]{babel}

\usepackage[left=2.61cm,right=2.61cm,top=2.72cm,bottom=2.72cm]{geometry}
\usepackage[hyperpageref]{backref}

\numberwithin{equation}{section}
\newtheorem{theorem}{Theorem}[section]
\newtheorem{lemma}[theorem]{Lemma}

\newtheorem{proposition}[theorem]{Proposition}

\newcommand{\eps}{\varepsilon}

\title{Soliton dynamics for the Schr\"odinger-Newton system}

\author[P.\ d'Avenia]{Pietro d'Avenia}
\address{Dipartimento di Meccanica, Matematica e Management
\newline\indent 
Politecnico di Bari
\newline\indent
Via Orabona 4,  70125  Bari, Italy}
\email{p.davenia@poliba.it}

\author[M.\ Squassina]{Marco Squassina}
\address{Dipartimento di Informatica
\newline\indent
Universit\`a degli Studi di Verona
\newline\indent
C\'a Vignal 2, Strada Le Grazie 15, 37134 Verona, Italy}
\email{marco.squassina@univr.it}

\thanks{Both the authors were supported by 2009 MIUR project:
   ``Variational and Topological Methods in the Study of Nonlinear Phenomena''. Moreover, the first author was supported also by the 2012 INDAM project ``Metodi Variazionali e Problemi Ellittici Non Lineari'' and 2011 FRA project ``Onde solitarie''. This work has been partially carried out during a stay of P. d'Avenia in Verona. He would like to express his deep gratitude to the Dipartimento di Informatica for the warm hospitality.}

\begin{document}

\begin{abstract}
	We investigate the soliton dynamics for the Schr\"odinger-Newton system
	by proving a suitable modulational stability estimates in the spirit of
	those obtained by Weinstein for local equations.
\end{abstract}

\maketitle

\section{Introduction}
\noindent
Let us consider the Schr\"odinger-Newton system
\begin{equation}
\label{eq:SNS1}
\begin{cases}
 i\hbar \partial_t u=
-\frac{\hbar^2}{2m} \Delta u + V(x) u - \phi u,  &\\
-\Delta \phi= 4\pi\gamma |u |^2,     &
\end{cases}
\end{equation}
where
$u:[0,\infty)\times\mathbb{R}^3\to\mathbb{C}$, 
$\phi:\mathbb{R}^3\to\mathbb{R}$ is 
the gravitational potential, $V:\mathbb{R}^3\to\mathbb{R}$ is 
an external potential, $m$ is the mass of the particle, 
$\gamma=Gm^2$, where $G$ is the Newton constant.
Up to suitable rescalings, \eqref{eq:SNS1}
%
%
%
%
can be rewritten as the equation
\begin{equation}
\label{eq:eqn}
i\eps \partial_t u^\eps= -\frac{\eps^2}{2} \Delta u^\eps + V(x) u^\eps 
- \frac{1}{\eps^2} \Big(\frac{1}{|x|}* |u^\eps |^2\Big)u^\eps
\quad\,\,
\hbox{in }
[0,\infty)\times\mathbb{R}^3.
\end{equation}
This equation was originally elaborated by
Pekar \cite{pekar} around 1954 in the framework of quantum mechanics.
Subsequently, in 1976, Choquard \cite{Lieb} adopted the equation as an approximation of the Hartree-Fock
theory. More recently, in 1996, Penrose \cite{penrose} settled it as a model of self-gravitating matter.
From the point of view of global well-posedness and smoothness for arbitrary initial 
data $u^\eps_0\in H^1({\mathbb R}^3,\mathbb{C})$, the Cauchy problem associated with 
\eqref{eq:eqn} was completely investigated in \cite{cazenave}. Concerning the existence and qualitative
properties of the associated standing wave solutions, we refer the reader to the classical
contributions by Lions \cite{Lions1,Lions2} on concentration compactness (see also \cite{morozschaft} for a more general situation). For what regards orbital stability
of solutions to \eqref{eq:eqn} -- for a fixed $\eps$ -- and with respect to a suitable family of ground states, 
we refer to the contribution due to Cazenave and
Lions \cite[Theorem IV.2]{cazlio} and those by Grillakis and Shatah \cite{gril1,gril2}. 
Years later, in the frame of stability theory for local Schr\"odinger
equation
\begin{equation}
\label{local-eq}
i\eps \partial_t u^\eps= -\frac{\eps^2}{2} \Delta u^\eps + V(x) u^\eps 
- |u^\eps |^{2p}u^\eps
\quad\,\,
\hbox{in }
[0,\infty)\times\mathbb{R}^3,\quad 0<p<\frac{2}{3},
\end{equation}
several contributions appeared about the study of the so called semi-classical
(or point particle) limit behaviour as the parameter $\eps$ vanishes, both for the standing waves
and the full evolutionary problem. Concerning the former, for local equations we refer to the monograph 
by Ambrosetti and Malchiodi \cite{ambr-malch} and to the references therein,
while for nonlocal equations, we refer to \cite{CSS} and to the related references. About the latter,
rigorous results about the soliton dynamics of local Schr\"odinger were obtained in various
papers, among which we mention the contributions 
by Bronski and Jerrard \cite{Bro-Jerr} and Keraani \cite{keraa} by means of arguments which are purely based
on the use of conservation laws satisfied by the equation and by the associated Newtonian
system $\ddot x(t)=-\nabla V(x(t))$, combined with the modulational stability estimates due to Weinstein
\cite{weinstein1,weinstein2}. With different techniques similar results were obtained in \cite{frolich1}
by Fr\"ohlich, Gustafson, Jonsson and Sigal (see also \cite{DV}). Roughly speaking, the soliton dynamics occurs when, choosing a suitable 
initial datum $u^\eps_0(x)=r((x-x_0)/\eps)$ the corresponding solution $u^\eps(t)$ mantains the shape
$r((x-x(t))/\eps)$, up to an estimable error and locally in time, in the transition from quantum
to classical mechanics, namely as $\eps\to 0$. For a nice survey on solitons and their stability features, see 
the work by Tao \cite{tao-solitons}. In the nonlocal case, the semiclassical limit
of the standing waves of \eqref{eq:eqn} was recently studied by Wei and Winter \cite{WW}. The full evolution problem
\eqref{eq:eqn} was studied in a soliton dynamics regime by Fr\"ohlich, Tsai and Yau in \cite{froltsaiyau}
along the line followed in \cite{frolich1} for the local case. 
On the contrary, to the best of the authors' knowledge,
there is no nonlocal counterpart of the study of point particle dynamics along the technique
initiated in the work by Bronski and Jerrard \cite{Bro-Jerr}. This is precisely the aim of this paper.
Let $r\in H^1(\mathbb{R}^3)$ be the unique radial, positive solution of 
\begin{equation}
\label{eq:lp}
-\frac{1}{2} \Delta r + r 
-  \Big(\frac{1}{|x|}* |r |^2\Big)r=0.
\end{equation}
The main tool exploited in \cite{Bro-Jerr,keraa} in the local case \eqref{local-eq}
is a kind of coercivity estimate for the differences ${\mathcal E}(\phi)-{\mathcal E}(r)$ 
upon suitable complex-valued functions $\phi$ such that $\|\phi\|_2=\|r\|_2$
for the energy functional ${\mathcal E}(\phi)=\|\nabla \phi\|_2^2/2 - \| \phi \|_{2p+2}^{2p+2}/(2p+2)$ associated 
with $- \Delta \phi /2 + \phi=|\phi|^{2p}\phi$ on $\mathbb{R}^3$, obtained by exploiting
the spectral properties of its linearized operator. 
The first main result of 
the paper is the validity of this property for the nonlocal equation \eqref{eq:lp}.
Precisely, let $\mathcal E:H^1(\mathbb{R}^3,\mathbb{C})\to\mathbb{R}$ be 
the energy functional defined by
\[
\mathcal{E}(\phi)=\frac{1}{2} \int | \nabla \phi|^2
- \frac{1}{2} \iint \frac{|\phi(x)|^2|\phi(y)|^2}{|x-y|}
\]
and $\|\cdot\|$ denote the $H^1(\mathbb{R}^3,\mathbb{C})$-norm.
Then we have the following
\begin{theorem}
\label{thm:enconv}
There exists a positive constant $C$ such that
\[
\mathcal{E} (\phi) - \mathcal{E} (r) \geq C
\inf_{x\in\mathbb{R}^3,\,\theta\in [0,2\pi)}
\| \phi - e^{i\theta} r(\cdot - x) \|^2
+ o \Big(
\inf_{x\in\mathbb{R}^3,\theta\in [0,2\pi)}
\| \phi - e^{i\theta} r(\cdot - x) \|^2 \Big),
\]
for every $\phi\in H^1(\mathbb{R}^3,\mathbb{C})$ such that 
$\| \phi \|_2 = \| r \|_2$ and $\!\!\inf\limits_{x\in\mathbb{R}^3,\,\theta\in [0,2\pi)}
\| \phi - e^{i\theta} r(\cdot - x) \|\leq \|r\|$.
\end{theorem}
\noindent
By combining Lions's concentration-compactness \cite{Lions1,Lions2} with
\cite[(ii) of Theorem IV.1]{cazlio} and recalling the uniqueness of the {\em ground state} $r$, an equivalent formulation of Theorem~\ref{thm:enconv} 
could be given by dropping the $o(\cdot)$ term and adding instead the requirement that the 
difference $\mathcal{E} (\phi) - \mathcal{E} (r) $ be small enough. For local Schr\"odinger
equations with power nonlinearity, Theorem~\ref{thm:enconv} was proved in \cite{weinstein1,weinstein2}
while \cite{MPS} contains a proof for the result for one dimensional Schr\"odinger systems.
We shall prove Theorem~\ref{thm:enconv} in Section~\ref{proof-thm1} by virtue of a careful study
of the (real and imaginary) linearized operators $L_-$ and $L_+$ associated with \eqref{eq:lp}
on some subspaces of $H^1(\mathbb{R}^3,\mathbb{C})$ defined by suitable
orthogonality conditions. Once the estimate of Theorem~\ref{thm:enconv} holds true, a natural
application is to obtain the soliton dynamics behaviour, in the semi-relativistic limit $\eps\to 0$,
for the Cauchy problem
\begin{equation}
\label{eq:CP}
\left\{
\begin{array}{l}
\displaystyle i \eps \partial_t u^\eps= - \frac{\eps^2}{2} \Delta u^\eps 
+ V(x) u^\eps - \frac{1}{\eps^2}  
\Big(\frac{1}{|x|}* |u^\eps |^2\Big)u^\eps\\
u^\eps(0,x)=\displaystyle r\Big(\frac{x-x_0}{\eps} \Big) 
e^{\frac{i}{\eps} x\cdot v_0},
\end{array}
\right.
\end{equation}
where 
and $x_0\in\mathbb{R}^3$ and $v_0\in\mathbb{R}^3$ 
are, respectively, the initial position and velocity of
\begin{equation}
\label{eq:dynsys2}
\left\{
\begin{array}{l}
\dot{x}(t)=v(t),  \\
\dot{v}(t)= - \nabla V(x(t)), \\
x(0)=x_0,  \\
v(0)=v_0.
\end{array}
\right.
\end{equation}
Problem \eqref{eq:CP} is globally well-posed, provided that 
$V\in L^m(\mathbb{R}^3)+L^\infty(\mathbb{R}^3)$, for some $m>3/2$ \cite[Corollary 6.1.2 and Example 1]{cazenave}.
Denoting $\|\cdot\|_{{\mathcal H}_\eps}^2=
\frac{1}{\eps}\|\nabla\cdot\|_2^2+\frac{1}{\eps^3}\|\cdot\|_2^2$, we prove the following
\begin{theorem}
\label{thm:dyn}
Assume that $V=V_1+V_2$ with $V_1\in {\mathcal C}^3(\mathbb{R}^3)$ and $D^2V_2\in C^2(\mathbb{R}^3)$,
where $V_2$ is bounded from below. Therefore, for every $\eps$ small, we have 
\[
\Big\|u^\eps(t,x)-r\Big(\frac{x-x(t)}{\eps}\Big) 
e^{i\frac{v(t)\cdot x}{\eps}}\Big\|_{{\mathcal H}_\eps}={\mathcal O}(\eps),
\]
on finite time intervals.
\end{theorem}
\noindent
We shall prove Theorem~\ref{thm:dyn} in Section~\ref{pf-del-thm2}
by showing a few preliminary facts about the energy expansion and the momentum identity 
for \eqref{eq:CP} and then exploiting Theorem~\ref{thm:enconv}
on a suitable auxiliary function related to the solution of \eqref{eq:CP}. Once that stage is achieved, the 
argument to get the uniform bound on the error -- on finite time intervals -- follows as in \cite{Bro-Jerr,keraa}.
Quite recently, Benci, Ghimenti and Micheletti in \cite{BGM1,BGM2} obtained, for a variant of the local equation
\eqref{local-eq}, a soliton dynamics behaviour with error estimate on the whole $[0,\infty)$ and, in general, working for
equations whose ground states need not be unique or nondegenerate. In a forthcoming paper, we aim to use their 
technique on a general nonlocal problem for which uniqueness and nondegeneracy results are not available yet.

\subsection*{Notations}
\begin{enumerate}
\item If $u,v \in\mathbb{C}$, $u\cdot v=\operatorname{Re}(u\bar{v})=\frac{1}{2}(u\bar{v}+v\bar{u})$.
\item $H^1(\mathbb{R}^3)=H^1(\mathbb{R}^3,\mathbb{R})$ and 
$H^1(\mathbb{R}^3,\mathbb{C})$ are the Sobolev spaces endowed 
with the norm 
$\|\cdot\|=(\|\cdot\|_2^2
+\frac{1}{2}\|\nabla\cdot\|_2^2)^{1/2}$.
\item If $u,v \in H^1(\mathbb{R}^3,\mathbb{C})$ we denote with 
$(u,v)$ the scalar product in $L^2(\mathbb{R}^3,\mathbb{C})$ and with $(u,v)_{H^1}=(u,v)+\frac{1}{2}(\nabla u,\nabla v)$.
\item ${\mathcal C}^3(\mathbb{R}^3)$ is the space of functions 
$u\in C^3(\mathbb{R}^3)$ with $\|D^{\alpha}u\|_\infty<\infty$
for any $|\alpha|\leq 3$.
\item $C$ denotes a generic positive constant which can changes 
from line to line.
\end{enumerate}

\section{Proof of Theorem \ref{thm:enconv}}
\label{proof-thm1}

\subsection{Preliminary tools} 
In this section we collect a few basic properties about the 
ground state solutions to 
\eqref{eq:lp} and its corresponding linearized operator.

\subsubsection{The limit problem}
Let us consider the eigenvalue problem
\begin{equation}
\label{eq:lpeignv}
-\frac{1}{2}\Delta \varphi - \Big(\frac{1}{|x|}* |\varphi |^2\Big)\varphi
= e \varphi.
\end{equation}
A fundamental tool in our analysis is the following result due to Lieb \cite[Theorem 8]{Lieb}.
\begin{theorem}
\label{dilieb}
If $\varphi\in H^1(\mathbb{R}^3)$, $\|\varphi\|_2=\lambda$ and 
$\mathcal{E}(\varphi)=\inf\{\mathcal{E}(\phi)\;\vline\; 
\phi\in H^1(\mathbb{R}^3), \|\phi\|_2=\lambda\}$,
then $\varphi$ satisfies equation \eqref{eq:lpeignv} 
for some $e<0$. 
Moreover if $\varphi\in H^1(\mathbb{R}^3,\mathbb{C})$ 
satisfies \eqref{eq:lpeignv} (not necessarily a minimizer) 
for an arbitrary Lagrange multiplier $e$, then:
\begin{enumerate}[label=(\roman*),ref=(\emph{\roman*})]
\item \label{it:ialieb}
$|x|^{-1}* |\varphi |^2\in L^p(\mathbb{R}^3)$,\,\,\,
for every $4\leq p \leq \infty$;
\item \label{it:iblieb}
$(|x|^{-1}* |\varphi |^2)\varphi\in L^p(\mathbb{R}^3,\mathbb{C})$,\,\,\,
for every $\frac{4}{3}\leq p \leq 6$;
\item \label{it:iilieb}
$|x|^{-1}* |\varphi |^2$ is a continuous function 
which goes to zero at infty;
\item \label{it:iiilieb}
If $e<0$ then $\varphi\in C^\infty (\mathbb{R}^3)$ and goes to 
zero at infinity (and hence $\varphi$ is a classical solution of 
\eqref{eq:lpeignv}).
\end{enumerate}
\end{theorem}
\noindent
Moreover we also need the following
\begin{proposition}
\label{groundfacts}
Let $r$ be the unique positive and radial solution of 
\eqref{eq:lp}.
We have that:
\begin{enumerate}[label=(\roman*),ref=(\emph{\roman*})]
\item \label{it:ndeg}
$r$ has a nondegenerate linearization (the linearization of 
\eqref{eq:lp} around $r$ has a nullspace that is entirely due 
to the equation’s invariance under phase and translation 
transformation);
\item \label{it:wei} 
$r(0)=\max_{x\in\mathbb{R}^3} r(x)$ and if we take 
$r(x)=r_0(|x|)$, we have that $r_0$ is 
strictly decreasing and
\[
\lim_{|x|\to \infty} r_0(|x|)e^{|x|}|x|=\lambda_0>0,
\qquad
\lim_{|x|\to \infty} \frac{r'_0(|x|)}{r_0(|x|)}=-1;
\]
\item \label{it:minpalla}
$r$ can be obtained as the minimum point of $\mathcal{E}$
in $\mathcal{M}=\{ u \in H^1(\mathbb{R}^3) \; 
\vline \; \| u \|_2 = \| r \|_2\}$.
\end{enumerate}
\end{proposition}
\begin{proof}
For the proof of \ref{it:ndeg} and \ref{it:wei} we refer to 
\cite{Lenz,Lieb,MZ,MT}. Here, for the sake of completeness,
we prove \ref{it:minpalla}. We know that for every $\alpha >0$,
$\mathcal{E}$ has a unique radial and strictly positive minimum 
point on $\{ u \in H^1(\mathbb{R}^3) \; 
\vline \; \| u \|_2 = \alpha\}$ 
(see \cite[Theorem 7 and Theorem 10]{Lieb}).
Let $\bar{u}$ be such minimum point 
on $\mathcal{M}$. There exists 
$\lambda>0$ such that
\[
-\frac{1}{2}\Delta \bar{u}  
-  \Big(\frac{1}{|x|}* |\bar{u} |^2\Big)\bar{u}
= - \lambda \bar{u}.
\]
It is easy to show that $\lambda^{-1} \bar{u} (\lambda^{-1/2}x)$ is 
a radial and strictly positive solution of \eqref{eq:lp}. Then, 
by the uniqueness, we have that 
$r(x)= \lambda^{-1} \bar{u} (\lambda^{-1/2}x)$
and, since $\|\bar{u}\|_2^2=\| r\|_2^2
=\|\bar{u}\|_2^2/\sqrt{\lambda}$, we get $\lambda=1$.
\end{proof}

\subsubsection{The linearized problem}
Let $r$ be the unique radial positive solution of \eqref{eq:lp} 
and consider the 
linearized operator $L$ for \eqref{eq:lp} at $r$, acting on 
$L^2(\mathbb{R}^3,\mathbb{C})$ with domain in 
$H^2(\mathbb{R}^3,\mathbb{C})$, 
\[
L\xi= - \frac{1}{2}\Delta \xi + \xi 
- \Big(\frac{1}{|x|} * r^2\Big) \xi 
- \Big(\frac{1}{|x|} * (r(\xi+\bar{\xi}))\Big) r .
\]
We can write
\[
L=
\left(
\begin{array}{cc}
L_+ & 0   \\
0   & L_- 
\end{array}
\right)
\]
where $L_+$ and $L_-$ act respectively on the real and 
imaginary part of $\xi$, i.e. if $\eta$ is real
\[
L_+ \eta = - \frac{1}{2}\Delta \eta + \eta 
- \Big(\frac{1}{|x|} * r^2\Big) \eta 
- 2 \Big(\frac{1}{|x|} * (r\eta)\Big) r 
\quad
\hbox{and}
\quad
L_- \eta = - \frac{1}{2}\Delta \eta + \eta 
- \Big(\frac{1}{|x|} * r^2\Big) \eta.
\]
It can be proved (see \cite{Lenz}) that
\begin{align}
\label{eq:kerL+}
\operatorname{Ker} L_+ =  & \operatorname{span} 
\left\{ \partial_{x_1} r, \partial_{x_2} r, 
\partial_{x_3} r \right\},\\
\label{ker-meno}
\operatorname{Ker} L_- 
= & \operatorname{span} \left\{ r \right\}.
\end{align}

\subsection{Preliminary results}
Let us set
\begin{equation*}
\Xi_j(r):=\partial_{x_j}((|x|^{-1} * r^2)r)
=(|x|^{-1} * r^2) \partial_{x_j} r + 2 (|x|^{-1} * (r\partial_{x_j} r))r,
\quad\,\, \text{for $j=1,2,3$}.
\end{equation*}
Notice that 
$\Xi_j(r) \in L^2(\mathbb{R}^3).$
Indeed,  for the first term it is enough to observe that
$|x|^{-1} * r^2 \in L^\infty (\mathbb{R}^3)$, by (i) of Theorem~\ref{dilieb}. 
Moreover, writing $|x|^{-1}=h_1+h_2$
with $h_1\in L^\infty (\mathbb{R}^3)$ and 
$h_2\in L^{3/2} (\mathbb{R}^3)$ yields
\begin{equation*}
\| (|x|^{-1} * (r\partial_{x_j} r)) r\|_2^2
\leq 
2\| h_1 \|_\infty^2 \| r\|_2^4 
\|\partial_{x_j} r \|_2^2
+2 \| h_2 \|_{3/2}^2 \| r\|_6^4 
\|\partial_{x_j} r \|_2^2.
\end{equation*}

\vskip2pt
\noindent
We shall prove the following

\begin{proposition}
\label{prop:L+}
Let $w\in H^1(\mathbb{R}^3,\mathbb{C})$ and $u$ and $v$ be the real
and the imaginary part of $w$. Let us assume that $\| w + r \|_2 = \| r \|_2$
and
\begin{equation}
(u,\Xi_j(r))=0,\quad\,\, \text{for $j=1,2,3$}.
\label{eq:211}
\end{equation}
Then, there exist positive constants $D, D_h$ such that
\begin{equation}
\label{eq:lowest}
(L_+ u, u) \geq D \| u \|^2 - D_1 \| w \|^4 - D_2 \| w \|^3.
\end{equation}
\end{proposition}
\noindent In order to prove Proposition \ref{prop:L+} we proceed by proving some preliminary results.
Let us set
$$
\mathcal{V}=\left\{ u\in H^1(\mathbb{R}^3) \;
\vline\; (u,r)=0\right\}.
$$

\begin{lemma}
	\label{zeroinf}
$\inf\limits_{\mathcal{V}} (L_+ u,u)=0$.
\end{lemma}

\begin{proof}
Since $r$ is the minimum point of
\[
I(u)=\mathcal{E}(u)+\| u \|_2^2
\]
on $\mathcal{M}$ (defined in \ref{it:minpalla} of Proposition
\ref{groundfacts}), then, for every smooth curve 
$\varphi:[-1,1]\to \mathcal{M}$ such that $\varphi(0)=r$ we 
have that
\[
\left.\frac{d^2 I(\varphi(s))}{ds^2}\right|_{s=0} \geq 0.
\]
Therefore, being $I'(r) = 0$, we get
\[
0 \leq
\langle I''(r) \varphi'(0), \varphi'(0)\rangle 
=  2 (L_+ \varphi'(0), \varphi'(0))
\]
Since the map $s \to \| \varphi(s) \|_2$ is constant, we have 
that $\varphi'(0) \in \mathcal{V}$. Then, by the arbitrariness 
of $\varphi'(0)$, we can say that 
$\inf_{\mathcal{V}} (L_+ u,u)\geq 0$.
On the other hand, for every $j=1,2,3$ we have that 
$\partial_{x_j} r \in \mathcal{V}$ 
and $(L_+ \partial_{x_j} r,\partial_{x_j} r)=0$ 
and then we conclude.
\end{proof}

\noindent

\begin{lemma}
	\label{4boundedness}
There exists $C>0$ such that
\begin{eqnarray}
& \displaystyle \int (|x|^{-1} * r^2) u^2 
\leq C \| u \|_2^2,& \text{for all $u\in L^2(\mathbb{R}^3)$,}  \label{eq:ec1} \\
& \displaystyle\int (|x|^{-1} * (r u)) r u 
\leq C  \| u \|_2 ^2 , & \text{for all $u\in L^2(\mathbb{R}^3)$.} \nonumber
\end{eqnarray}
\end{lemma}
\begin{proof}
By \ref{it:ialieb} of Theorem \ref{dilieb}, inequality \eqref{eq:ec1} easily follows. Moreover, combining the Hardy-Littlewood-Sobolev 
\cite[Theorem 4.3]{liebloss} and H\"older inequality, we get
\begin{equation*}
\int (|x|^{-1} * (r u)) r u  
\leq  C \|ru\|_{6/5} \|ru\|_{6/5}\\
\leq   C\|r\|_3^2 \|u\|_2^2 \leq  C  \|u\|_2^2,
\end{equation*}
concluding the proof.
\end{proof}

\noindent
Moreover, we have the following
\begin{lemma}
	\label{convNonl}
Assume that $u_n \rightharpoonup u$ in $H^1(\mathbb{R}^3)$ as $n\to\infty$. Then, up to a subsequence, we have
\begin{align}
 \lim_n\int (|x|^{-1} * r^2) u_n^2 &=
\int (|x|^{-1} * r^2) u^2, \label{prima} \\	
 \lim_n\int (|x|^{-1} * (r u_n)) r u_n &=
\int (|x|^{-1} * (r u)) r u. \label{seconda}
\end{align}
\end{lemma}
\begin{proof}
	Up to a subsequence, $u_n \to u \hbox{ a.e.}$
	Since the sequence $\{u_n^2\}$ is bounded in $L^{6/5}(\mathbb{R}^3)$,
	up to a subsequence, it converges weakly to some $z\in L^{6/5}(\mathbb{R}^3)$. 
	Taking into account the poinwise convergence of $\{u_n\}$ to $u$, it follows that $z=u^2$.
	Hence, in order to get~\eqref{prima}, it is sufficient to have $|x|^{-1}*r^2\in L^6(\mathbb{R}^3)$
	which follows from \ref{it:ialieb} of Theorem \ref{dilieb} Concerning~\eqref{seconda}, we have
	\begin{equation*}
	\Big|	\int (|x|^{-1} * (r u_n)) r u_n- 
		\int (|x|^{-1} * (r u)) r u \Big|\leq 	I_n+J_n,
	\end{equation*}
	where we have set
	$$
	I_n=\Big|	\int (|x|^{-1} * (r u_n)) (r u_n-ru) \Big|,\quad\,\,\,
	J_n=\Big|	\int (|x|^{-1} * (r u_n-r u)) r u \Big|.
	$$
	Observe that, since $\{u_n^{6/5}\}$ 
	converges weakly to $u^{6/5}$ in $L^2(\mathbb{R}^3)$ and $r^{6/5}\in L^2(\mathbb{R}^3)$, we have	$\| r u_n\|_{6/5}\to \| r u\|_{6/5}$.
	Since $ru_n\rightharpoonup ru$ in $L^{6/5}(\mathbb{R}^3)$ as $n\to\infty$, the uniform convexity
	of $L^{6/5}(\mathbb{R}^3)$ yields $\|r u_n- ru\|_{6/5}\to 0$ 
	as $n\to\infty$. Therefore, from the Hardy-Littlewood-Sobolev 
	inequality, we deduce
	\[
	I_n \leq C \|ru_n\|_{6/5} \|ru_n-ru\|_{6/5}  \to 0
	\quad\hbox{ and }\quad
	J_n \leq
	C\|ru_n-ru\|_{6/5}\|ru\|_{6/5}\to 0,
	\]
	which concludes the proof.
\end{proof}

\noindent
Let us set
\[
\mathcal{V}_0 
= \big\{ u\in H^1(\mathbb{R}^3) \; \vline \,\,\, (u,r)
=(u,\Xi_j(r))=0, \,\,
j=1,2,3 \big\}.
\]
\noindent
Concerning the coercivity of $L_+$ on $\mathcal{V}_0$, we have the following
\begin{lemma}
	\label{positconstr}
$\inf\limits_{u\in\mathcal{V}_0} \frac{(L_+ u,u)}{\|u\|^2}>0$.
\end{lemma}
\begin{proof}
We claim, first, that $\inf\limits_{u\in\mathcal{V}_0} \frac{(L_+ u,u)}{\|u\|_2^2}>0$.
To this aim, let us consider 
$$
\alpha:=\inf_{u\in\mathcal{V}_0,\|u\|_2=1} (L_+ u,u).
$$ 
We want to prove that $\alpha>0$.
Since $\mathcal{V}_0 \subset \mathcal{V}$, then $\alpha\geq 0$ in light of Proposition~\ref{zeroinf}. 
Suppose by contradiction that $\alpha=0$ and let $\{u_n\}\subset H^1(\mathbb{R}^3)$ be a 
minimizing sequence. By virtue of Lemma \ref{4boundedness}, 
we readily have that $\{u_n\}$ is bounded in 
$H^1(\mathbb{R}^3)$. Then there exists $u\in H^1(\mathbb{R}^3)$
such that, up to a subsequence, $u_n \rightharpoonup u \hbox{ in } H^1(\mathbb{R}^3)$ and $u\in \mathcal{V}_0$.
In turn, in light of Lemma~\ref{convNonl}, we deduce that
\begin{align*}
0 \leq  (L_+ u, u) & \leq \liminf_n \Big( \|u_n \|^2
- \int(|x|^{-1} * r^2) u_n^2
- 2 \int(|x|^{-1} * (r u_n)) r u_n \Big) \\
& =  \lim_n (L_+ u_n, u_n) = 0,
\end{align*}
so that $(L_+ u, u)=0$.
In turn, recalling that $(L_+ u_n, u_n)\to 0$ as $n\to\infty$, we get
\begin{align*}
\|u \|^2 \leq & \liminf_n  \|u_n \|^2
\leq \limsup_n  \|u_n \|^2 \\
= & \lim_n \Big( (L_+ u_n, u_n) 
+ \int(|x|^{-1} * r^2) u_n^2
+ 2 \int(|x|^{-1} * (r u_n)) r u_n \Big)\\
= & (L_+ u, u) 
+ \int (|x|^{-1} * r^2) u^2
+2 \int (|x|^{-1} * (r u)) r u
= \|u \|^2.
\end{align*}
Then $\{u_n\}$ converges to $u$ in $H^1(\mathbb{R}^3)$ 
and $u$ solves the constrained minimization problem. 
Then there exist five Lagrange multipliers 
$\lambda, \mu, \gamma_1, \gamma_2, \gamma_3\in\mathbb{R}$ such that
for every $\eta\in H^1(\mathbb{R}^3)$
\[
(L_+ u, \eta)=\lambda (u,\eta)+\mu (r,\eta) 
+ \sum_{j=1}^{3} \gamma_j \int \Xi_j(r)\eta.
\]
Since $(L_+u,u)=0$ and $u\in\mathcal{V}_0$, it follows immediately that $\lambda=0$.
We claim that, for every $h=1,2,3$,
\begin{equation*}
0=(L_+ u, \partial_{x_h} r) 
=  \sum_{j=1}^{3} \gamma_j 
\int \Xi_j(r)\partial_{x_h} r
=  \gamma_h 
\int \Xi_h(r)\partial_{x_h} r.
\end{equation*}
This follows by the following facts: $(r,\partial_{x_h} r)=0$, 
$L_+$ is a self-adjoint operator, 
$ \partial_{x_h} r \in \operatorname{Ker} L_+$,
$r\in H^2(\mathbb{R}^3)$ and 
for every $j\neq h$ it holds
\begin{equation}
\label{eq:inddif}
\begin{split}
\int \Xi_j(r)
\partial_{x_h} r
= &
\int \partial_{x_j}\Big(\Big(\frac{1}{|x|} * r^2\Big) r 
\Big) \partial_{x_h} r
= 
\int \partial_{x_j}\Big( -\frac{1}{2} \Delta r + r \Big) 
\partial_{x_h} r\\
= & 
\frac{1}{2} \int  \nabla \partial_{x_j}r \cdot \nabla \partial_{x_h}r 
+ \int \partial_{x_j} r \ \partial_{x_h} r\\
= & \frac{1}{2}
\int\frac{x_j x_h}{|x|^4}[(r''_0 (|x|) |x|)^2 - (r'_0 (|x|))^2]
+\int\frac{x_j x_h}{|x|^2} (r'_0 (|x|))^2=0.
\end{split}
\end{equation}
Moreover
\begin{equation}
\label{eq:indug}
\int \Xi_h(r)\partial_{x_h} r
= 
\int \partial_{x_h}\Big(\Big(\frac{1}{|x|} * r^2\Big) r 
\Big) \partial_{x_h} r
= 
\int \partial_{x_h}\Big( - \frac{1}{2}\Delta r + r \Big) 
\partial_{x_h} r
= 
\| \partial_{x_h} r\|^2.
\end{equation}
It follows that $\gamma_h=0$ for every $h=1,2,3$, yielding in turn
\begin{equation}
\label{eq:mu}
(L_+ u, \eta)=\mu (r,\eta),\quad\text{for every $\eta\in H^1(\mathbb{R}^3)$.} 
\end{equation}
Now we claim that $\mu \neq 0$. Indeed if we suppose by
contradiction that $\mu = 0$, then, from \eqref{eq:mu},
$u\in \operatorname{Ker} L_+$. Thus, from \eqref{eq:kerL+}, 
we have that $u=\beta \cdot \nabla r$ with 
$\beta=(\beta_1,\beta_2,\beta_3)\in\mathbb{R}^3$. 
Moreover, since $u\in \mathcal{V}_0$, then, using \eqref{eq:inddif} and \eqref{eq:indug},  we have
\[
0=\int \Xi_j(r)(\beta\cdot\nabla r)
= \beta_j \| \partial_{x_j} r \|^2,
\quad
\hbox{for every } j=1,2,3.
\]
Then $\beta=0$, namely $u=0$, contradicting $\|u\|_2=1$. Notice now that
\[
L_+(x\cdot\nabla r)=
-\Delta r + \sum_{j=1}^{3} x_j\partial_{x_j} \Big[ 
-\frac{1}{2}\Delta r + r -  \Big(\frac{1}{|x|}* r^2\Big)r \Big]
= -\Delta r
\]
and, furthermore,
\[
L_+(r)= -2 (|x|^{-1} * r^2)r.
\]
Then
\[
L_+\Big(-\frac{\mu}{2}(r+x\cdot\nabla r)\Big)=\mu r=L_+u.
\]
In turn, by the nondegeneracy of $r$ (see \eqref{eq:kerL+}), we learn that there exist 
$\vartheta=(\vartheta_1,\vartheta_2,\vartheta_3)\in\mathbb{R}^3$ with
$$
u=-\frac{\mu}{2}(r+x\cdot\nabla r)+\vartheta\cdot\nabla r.
$$
We want to show that $\vartheta=0$. Since $u\in\mathcal{V}_0$, for every $j=1,2,3$, we have
\begin{align*}
0=&\int\Xi_j(r) u 
= -\frac{\mu}{2} \int\Xi_j(r) r 
- \frac{\mu}{2}\int \Xi_j(r) x\cdot\nabla r
+ \vartheta_j \| \partial_{x_j} r \|^2
\end{align*}
where we have used \eqref{eq:inddif} and \eqref{eq:indug}.
On the other hand, we have
\begin{equation}
\label{eq:xir}
\int\Xi_j(r) r 
=3 \int\Big(\frac{1}{|x|} * r^2\Big)r\partial_{x_j} r
=3\int \Big(-\frac{1}{2}\Delta r + r\Big)\partial_{x_j} r
=\frac{3}{2}\int\nabla r \cdot \nabla (\partial_{x_j} r) 
+ 3\int r \partial_{x_j} r
=0
\end{equation}
and, since the map $x\mapsto (|x|^{-1} * r^2)r$  is radially symmetric,
\[
\int \Xi_j(r)
(x\cdot\nabla r)
=
\int \partial_{x_j} \left[\left(\frac{1}{|x|} * |r |^2\right)r
\right]
|x| r'_0(|x|)=0.
\]
Then, for every $j=1,2,3$, we get $\vartheta_j \| \partial_{x_j} r \|^2 = 0$, yielding
in turn $\vartheta=0$. Thus
\[
u=-\frac{\mu}{2}(r+x\cdot\nabla r).
\]
But, since $u\in\mathcal{V}_0$, 
\begin{equation}
\label{eq:scpr}
0=(u,r)=-\frac{\mu}{2}\left(\|r\|_2^2 
+ (x\cdot\nabla r,r)\right).
\end{equation}
Moreover, integrating by parts, we have
\[
(x\cdot\nabla r,r)=
\frac{1}{2}\sum_{h=1}^3 \int x_h \partial_{x_h} (r^2)
=-\frac{3}{2} \|r\|_2^2.
\]
Dropping in \eqref{eq:scpr} we get the contradiction and so
that the proof of the claim is complete. Then, there exists a positive constant
$\alpha_0 > 0$ such that
\begin{equation}
\label{eq:lbl2}
(L_+ u,u)\geq\alpha_0\|u\|_2^2,\,\,\quad \text{for every $u\in\mathcal{V}_0$.}
\end{equation}
If we put $|||u|||:=\sqrt{(L_{+}u,u)}$  
for $u\in \mathcal{V}_{0}$, it is readily checked that $|||\cdot|||$ 
satisfies the required properties of a norm. Furthermore,
if $\{u_n\}$ is a Cauchy sequence in  
$(\mathcal{V}_{0},|||\cdot|||)$, then, by \eqref{eq:lbl2},
$\{u_n\}$ strongly converges to a function $u$ in 
$L^2(\mathbb{R}^3)$ and $u\in\mathcal{V}_{0}$. Moreover, using
Lemma \ref{4boundedness}, we have that $\{u_n\}$ is a Cauchy
sequence in $H^1(\mathbb{R}^3)$ and then $u$ has to be 
necessarily the strong limit in $H^1(\mathbb{R}^3)$. Therefore,
$u_n \to u$ in $(\mathcal{V}_{0},|||\cdot|||)$ and so we get 
that $(\mathcal{V}_{0},|||\cdot|||)$ is a Banach space and
$|||\cdot|||$ is equivalent to the  
norm of $H^1(\mathbb{R}^3)$. This concludes the proof.
\end{proof}

\begin{proof}[Proof of Proposition~\ref{prop:L+}]
Since $\| w + r\|_2 = \| r \|_2$, we have that
\[
\| r \|_2^2=
\| r \|_2^2 + \| w \|_2^2 + 2 (u,r) 
\]
and then
\begin{equation}
(r,u) = - \frac{1}{2} \| w \|_2^2 =-\frac{1}{2} ( \| u \|_2^2 + \| v \|_2^2 ).
\label{eq:210}
\end{equation}
Without loss of generality, we can suppose that $\| r \|_2 = 1$.
Let us write $u=u_\parallel + u_\perp$ where $u_\parallel=(u,r)r$.
We notice that $u_\perp$ is orthogonal to $r$ in 
$L^2(\mathbb{R}^3)$
and, combining \eqref{eq:211} with
\eqref{eq:xir}
we have that
$
(u_\perp,\Xi_j(r))=0
$
and namely $u_\perp \in \mathcal{V}_0$.
Since $L_+$ is selfadjoint, we have that
\[
(L_+ u, u)=(L_+ u_\parallel, u_\parallel)
+2(L_+ u_\perp, u_\parallel)+(L_+ u_\perp, u_\perp).
\]
So we study separately each term in the right hand side. 
By \eqref{eq:210}, the selfadjointness of $L_+$ and since 
$r$ is solution of \eqref{eq:lp}, we have that
\begin{equation}
\label{eq:parpar}
(L_+ u_\parallel, u_\parallel)=\frac{1}{4} \| w \|_2^4 (L_+ r, r)
=-\frac{1}{2} \| r \|^2 \| w \|_2^4 
\end{equation}
and
\begin{equation}
\label{eq:parperp}
\begin{split}
(L_+ u_\perp, u_\parallel) = & 
- \frac{1}{2} \| w \|_2^2 (u_\perp, L_+ r) 
= \frac{1}{2}\| w \|_2^2 \int \nabla u_\perp \cdot \nabla r
= \frac{1}{2}\| w \|_2^2 \Big( \int \nabla u \cdot \nabla r 
- \int \nabla u_\parallel \cdot \nabla r \Big)\\
\geq & - \frac{1}{2}\| w \|_2^2 \| \nabla w \|_2 \| \nabla r \|_2.
\end{split}
\end{equation}
Finally we notice that
\[
\| \nabla u \|_2^2 \leq  
2( \| \nabla u_\parallel \|_2^2 + \| \nabla u_\perp \|_2^2)
=\frac{1}{2} \| w\|_2^4 \| \nabla r \|_2^2 + 2 \| \nabla u_\perp \|_2^2
\]
so that
\[
\| \nabla u_\perp \|_2^2 \geq \frac{1}{2} \| \nabla u \|_2^2
- \frac{1}{4} \| w\|_2^4 \| \nabla r \|_2^2.
\]
Then,
since $u_\perp \in \mathcal{V}_0$, applying Lemma \ref{positconstr} we have that
\begin{equation}
\label{eq:perpperp}
\begin{split}
(L_+ u_\perp, u_\perp ) \geq & C \| u_\perp \|^2 
= C( \| u \|_2^2 - \| u_\parallel \|_2^2 + \| \nabla u_\perp \|_2^2)\\
\geq & C \big( \| u \|_2^2 - | (u,r) |^2
+ \frac{1}{2} \| \nabla u \|_2^2 - \frac{1}{4} \| w \|_2^4 \| \nabla r \|_2^2\big)\\
\geq & C ( \| u \|^2 - \| w \|_2^4 ).
\end{split}
\end{equation}
Combining \eqref{eq:parpar}, \eqref{eq:parperp} and \eqref{eq:perpperp} we get
\eqref{eq:lowest}.
\end{proof}

\noindent
Concerning the coercivity of $L_-$, we have the following
\begin{proposition}
\label{prop:L-}
	$\inf\limits_{v\neq 0, \; (v,r)_{H^{1}}=0}
	\frac{\left(L_{-}v,v\right)}{\|v\|^{2}}>0.$ 
\end{proposition}
\begin{proof}
Up to arguing as in the end of the proof of Lemma~\ref{positconstr}, it is enough to prove that
$$
\inf\limits_{v\neq 0, \; (v,r)_{H^{1}}=0}
	\frac{\left(L_{-}v,v\right)}{\|v\|^{2}_{2}}>0.
$$
Let us first prove that $L_-$ is nonnegative. 
Since $|x|^{-1}\in L^{3-\delta}(\mathbb{R}^3)+L^{3+\delta}(\mathbb{R}^3)$ and 
$r^2\in L^{(3-\delta)'}(\mathbb{R}^3)\cap L^{(3+\delta)'}(\mathbb{R}^3)$ for $\delta>0$ 
small, by applying \cite[Lemma 2.20]{liebloss} 
we have that $x\mapsto |x|^{-1} * r^2$ goes to zero for $|x|\to\infty$ and so
\[
\lim_{|x|\to \infty} \big[1-(|x|^{-1} * r^2)\big]
=1.
\]
In turn, from~\cite[Theorem 3.1, p.165]{BerSchub}, we learn 
that $L_-$ is bounded from 
below and has a discrete spectrum over $(-\infty, 1)$ which 
consists of eigenvalues of finite multiplicity.
Moreover $r\in \operatorname{Ker} L_-$ and so $0$ is an 
eigenvalue of $L_-$
and $r$ is a corresponding eigenfunction.
But, from \cite[Theorem 3.4, p.179]{BerSchub}, since the smallest 
eigenvalue of 
$L_-$ is lower than $1$, then it is simple and the 
corresponding 
eigenfunction can be chosen to be positive everywhere. The 
positivity of
$r$ implies that $0$ is the smallest eigenvalue of 
$L_-$ and $r$ is the corresponding eigenfunction. Thus, for any 
$v \in H^1(\mathbb{R}^3)$, we have that $(L_- v,v)\geq 0$. 
Let us consider 
$$
\omega:=\inf_{(v,r)_{H^1}=0,\,\|v\|_2=1} (L_- v,v),
$$
and assume by contradiction that $\omega=0$. Let $\{v_n\}\subset H^1(\mathbb{R}^3)$ be a 
minimizing sequence. By virtue of \eqref{eq:ec1}, 
it follows that $\{v_n\}$ is bounded in 
$H^1(\mathbb{R}^3)$. Then there exists $v\in H^1(\mathbb{R}^3)$
such that, up to a subsequence, $v_n \rightharpoonup v$ in $H^1(\mathbb{R}^3)$ and $(v,r)_{H^1}=0$.
In turn, in light of~\eqref{prima}, we deduce that
\begin{equation*}
0 \leq  (L_- v, v)  \leq \liminf_n \Big( \|v_n \|^2
- \int(|x|^{-1} * r^2) v_n^2 \Big) 
 =  \lim_n (L_- v_n, v_n) = 0,
\end{equation*}
so that $(L_- v, v)=0$. Hence, we obtain
\begin{align*}
\|v \|^2 \leq & \liminf_n  \|v_n \|^2 
\leq \limsup_n  \|v_n \|^2 
=  \lim_n \Big( (L_- v_n, v_n) 
+ \int (|x|^{-1} * r^2) v_n^2\Big)\\ 
= & (L_- v, v) 
+ \int (|x|^{-1} * r^2) v^2
= \|v \|^2.
\end{align*}
Then $\{v_n\}$ converges to $v$ in $H^1(\mathbb{R}^3)$ which implies that $\|v\|_2=1$ 
and $v$ solves the minimization problem.
In turn, there exist two Lagrange multipliers $\lambda,\mu\in\mathbb{R}$ 
such that, for every $\eta\in H^1(\mathbb{R}^3)$, it holds
\begin{equation}
	\label{constr-eq}
(L_-v,\eta)=\lambda (v,\eta)+\mu (r,\eta)_{H^1}
\end{equation}
Then, dropping $\eta=v$ into~\eqref{constr-eq} immediately yields $\lambda=0$, so that, 
for any $\eta\in H^1(\mathbb{R}^3)$,
\begin{equation}
	\label{constr-eqbis}
(L_-v,\eta)=\mu (r,\eta)_{H^1} .
\end{equation}
Finally, by choosing now $\eta=r$ into equation~\eqref{constr-eqbis}, and recalling that $L_- r=0$ yields
$$
0=(v,L_-r)=(L_- v,r)=\mu \| r\|^2
$$
where we used the fact that $L_-$ is self-adjoint.
Then $\mu=0$, namely $L_- v=0$. In light of~\eqref{ker-meno}, there is
$\vartheta\in\mathbb{R}\setminus\{0\}$ with $v=\vartheta r$.
Thus $0= \vartheta \| r\|^2$ 
that is a contradiction. Then $\omega>0$ and the proof is complete.
\end{proof}

\noindent
For the proof of Theorem~\ref{thm:enconv} we shall also need the following

\begin{lemma}
\label{existence-inf}
Let $\phi\in H^1(\mathbb{R}^3,\mathbb{C})$ with 
$\| \phi \|_2 = \| r \|_2$ and $\inf\limits_{x\in\mathbb{R}^3,\,\theta\in [0,2\pi)}
\| \phi - e^{i\theta} r(\cdot - x) \|\leq \|r\|$. Then
$
\inf\limits_{x\in\mathbb{R}^3,\,\theta\in [0,2\pi)}
\| \phi - e^{i\theta} r(\cdot - x)\|^2
$
is achieved at some $x_0\in\mathbb{R}^3$ and $\gamma\in [0,2\pi)$.
\end{lemma}
\begin{proof}
Consider the function $\Upsilon:\mathbb{R}^3\times [0,2\pi)\to\mathbb{R}$ defined by setting
$$
\Upsilon(x,\theta)=\| \phi-e^{i\theta} r(\cdot-x)\|^2
$$
It is readily checked that $\Upsilon$ is continuous. Moreover,
since $\| \phi \|_2 = \| r \|_2$, we get
$$
\Upsilon(x,\theta)=2\|r\|_2^2+\frac{1}{2}\| \nabla r\|_2^2
-2\mathfrak{Re} \int e^{i\theta}\bar \phi(y)r(y-x)dy - \mathfrak{Re}\int e^{i\theta}\nabla\bar \phi(y)\cdot \nabla r(y-x)dy
+\frac{1}{2}\| \nabla \phi\|_2^2.
$$
Taking into account that the families of functions
$(r(\cdot-x))_{x\in\mathbb{R}^3}$ and $(\nabla r(\cdot-x))_{x\in\mathbb{R}^3}$ are bounded in $L^2(\mathbb{R}^3)$
and converge pointwise (almost everywhere) to zero as $|x|\to\infty$, it follows that they converge 
weakly to zero in $L^2(\mathbb{R}^3)$ as $|x|\to\infty$. In turn, it readily follows that, for any $\theta\in[0,2\pi)$,
$$
\lim_{|x|\to\infty}\Upsilon(x,\theta)=2\|r\|_2^2+\frac{1}{2}\| \nabla r\|_2^2
+\frac{1}{2}\| \nabla \phi\|_2^2>\|r\|^2.
$$
On the other hand, in light of the second assumption on the function
$\phi$, for every $\delta>0$, there exist points $\tilde x\in\mathbb{R}^3$
and $\tilde\theta\in [0,2\pi)$ such that $\Upsilon(\tilde x,\tilde\theta)\leq \|r\|^2+\delta$.
It follows that the infimum of $\Upsilon$ over the unbounded set $\mathbb{R}^3\times [0,2\pi)$ coincides
with the infimum of $\Upsilon$ over the compact set $\bar{B}_R(0)\times [0,2\pi]$ for every $R>0$ sufficiently large,
yielding in turn the desired conclusion.
\end{proof}

\begin{proof}[Proof of Theorem \ref{thm:enconv} concluded]
Let $\phi\in H^1(\mathbb{R}^3,\mathbb{C})$ be a function such that
$\| \phi \|_2 = \| r \|_2$ and
$\inf_{x\in\mathbb{R}^3,\,\theta\in [0,2\pi)}
\| \phi - e^{i\theta} r(\cdot - x) \|\leq \|r\|.$ 
In light of Lemma~\ref{existence-inf}
there exist $x_0 \in \mathbb{R}^3,\gamma\in[0,2\pi)$ with
$$
\inf\limits_{x\in\mathbb{R}^3,\,\theta\in [0,2\pi)}
\| \phi - e^{i\theta} r(\cdot - x)\|^2=\| \phi - e^{i\gamma} r(\cdot - x_0)\|^2.
$$
Let us set 
$w(x):=e^{-i\gamma}\phi(x+x_0)-r(x)$.
Denoting by $u$ and $v$ respectively the real 
and the imaginary part of $w$, we claim that 
$u$ satisfies $(u,\Xi_j(r))=0$ for $j=1,2,3$
and $(v,r)_{H^1}=0$.
Indeed, if, as in the proof of Lemma \ref{existence-inf}, for any $\phi \in H^1(\mathbb{R}^3,\mathbb{C})$, 
$x\in\mathbb{R}^3$, $\theta\in \mathbb{R}$ we consider
\begin{align*}
\Upsilon(x,\theta)=&\| \phi-e^{i\theta} r(\cdot-x)\|^2\\
=&\|\phi\|^2 + \| r \|^2
-2\mathfrak{Re} \int e^{i\theta}\bar \phi(y)
\Big(-\frac{1}{2}\Delta r + r\Big)(y-x)dy \\
=&\|\phi\|^2 + \| r \|^2
-2\mathfrak{Re} \int e^{i\theta}\bar \phi(y)
[(|\cdot |^{-1} * r^2)r](y-x)dy ,
\end{align*}
we have
\[
\frac{\partial \Upsilon}{\partial x_j} (x,\theta)
=2\mathfrak{Re} \int e^{i\theta}\bar \phi(y+x)
\Xi_j(r)(y)dy
\quad\hbox{and}\quad
\frac{\partial \Upsilon}{\partial \theta} (x,\theta)
=2\mathfrak{Im} \int  e^{i\theta}\bar \phi(y+x)
\Big(-\frac{1}{2}\Delta r + r\Big)(y)dy.
\]
If $x=x_0$ and $\theta=\gamma$, since
$e^{i\gamma}\bar \phi(\cdot+x_0)=\bar{w} + r$, 
$\partial_{x_j} \Upsilon(x_0,\gamma) = 0$
and $\partial_\theta \Upsilon (x_0,\gamma) = 0$,
using \eqref{eq:xir}, we get the orthogonality conditions.
Then we consider the action 
$I(\phi)= \mathcal{E} (\phi) + \|\phi\|_2^2$
and we control the norm of $w$ in terms of the difference
$ I(\phi) - I(r)$.
Using the scale invariance of $I$, 
recalling that $\langle I'(r),w \rangle = 0$
and using also
\begin{align*}
\langle I''(\zeta)\varsigma,\varsigma\rangle
= & 2 \| \varsigma \|^2
-2 \int (|x|^{-1} * |\zeta|^2 ) |\varsigma|^2\\
& -4 \Big[
\int (|x|^{-1} * \mathfrak{Re}\zeta\mathfrak{Re}\varsigma )
\mathfrak{Re}\zeta\mathfrak{Re}\varsigma
+2 \int (|x|^{-1} * \mathfrak{Re}\zeta\mathfrak{Im}\varsigma )
\mathfrak{Im}\zeta\mathfrak{Re}\varsigma
+\int (|x|^{-1} * \mathfrak{Im}\zeta\mathfrak{Im}\varsigma )
\mathfrak{Im}\zeta\mathfrak{Im}\varsigma
\Big]
\end{align*}
for $\zeta,\varsigma\in H^1({\mathbb R}^3,\mathbb{C})$,
the orthogonality conditions proved before,
Propositions \ref{prop:L+} and \ref{prop:L-}, and the 
Hardy-Littlewood-Sobolev inequality we have
\begin{align*}
I(\phi) - I(r) = & I(r + w) - I(r) 
= \langle I'(r),w \rangle 
+ \frac{1}{2}\langle I''(r+\vartheta w)w,w \rangle\\
= & \| w \|^2 
- \int (|x|^{-1} * |r+\vartheta w|^2) |w|^2
- 2 \int (|x|^{-1} * (r+\vartheta u)u) 
(r+\vartheta u)u\\
&
- 2 \vartheta^2 \int (|x|^{-1} * v^2 ) v^2 
- 4 \vartheta \int (|x|^{-1} * (r+\vartheta u)v)uv\\
= &
(L_+ u,u) + (L_- v, v) 
- 2 \vartheta \int (|x|^{-1} * r u)|w|^2
- \vartheta^2 \int (|x|^{-1} * |w|^2)|w|^2\\
&
- 4 \vartheta \int (|x|^{-1} * r u)u^2
- 2 \vartheta^2 \int (|x|^{-1} * u^2)u^2
- 2 \vartheta^2 \int (|x|^{-1} * v^2)v^2\\
&
- 4 \vartheta \int (|x|^{-1} * r v)u v
- 4 \vartheta^2 \int (|x|^{-1} * u v )u v\\
\geq &
D \| w \|^2 - D_1 \| w \|^4 - D_2 \| w \|^3,
\end{align*}
which concludes the proof.
\end{proof}

\section{Proof of Theorem~\ref{thm:dyn}}
\label{pf-del-thm2}

\subsection{Preliminary results}
Let $u^\eps$ be a solution of the Cauchy problem \eqref{eq:CP}. 
The energy is defined as
$$
E_\eps(t)=\frac{1}{2\eps}\int |\nabla u^\eps(t,x)|^2+\frac{1}{\eps^3}\int V(x)|u^\eps(t,x)|^2-
\frac{1}{2\eps^5}\iint \frac{|u^\eps(t,x)|^2|u^\eps(t,y)|^2}{|x-y|}
$$
and $E_\eps(t)=E_\eps(0)$ for every $t\geq 0$.
Moreover the mass conservation reads as
\[
\frac{1}{\eps^3} \int |u^\eps(t,x)|^2
=\|r\|_2^2=:m,
\qquad t\geq 0, \;\eps>0.
\] 
For both conservations we refer the reader to \cite[Theorem 4.3.1]{cazenave}.
Setting
$$
{\mathcal H}(t):=\frac{1}{2}m|v(t)|^2+mV(x(t)),\quad t\geq 0,
$$
from system \eqref{eq:dynsys2} it follows  that ${\mathcal H}(t)={\mathcal H}(0)$, for all $t>0$. We have the following

\begin{lemma}
\label{espansione-energia}
$E_\eps(t)={\mathcal E}(r)+{\mathcal H}(t)+{\mathcal O}(\eps^2)$
for all $t\in [0,\infty)$ and $\eps>0$.
\end{lemma}
\begin{proof}
First, we observe  that
$$
\Big|\nabla\Big(r\Big(\frac{x-x_0}{\eps} \Big) 
e^{\frac{i}{\eps} x\cdot v_0}  \Big)\Big|^2=\frac{1}{\eps^2}\big|\nabla r\Big(\frac{x-x_0}{\eps} \Big)\Big|^2 
+\frac{|v_0|^2}{\eps^2}r^2\Big(\frac{x-x_0}{\eps} \Big). 
$$
Then, by the conservation of energy $E_\eps$, for any $t\in [0,\infty)$ and $\eps>0$, there holds
\begin{align*}
E_\eps(t)= E_\eps(0)&=\frac{1}{2\eps^3}\int \big|\nabla r\Big(\frac{x-x_0}{\eps} \Big)\Big|^2
+\frac{|v_0|^2}{2\eps^3}\int r^2\Big(\frac{x-x_0}{\eps} \Big) \\
& +\frac{1}{\eps^3}\int V(x)r^2\Big(\frac{x-x_0}{\eps}\Big)-
\frac{1}{2\eps^5}\iint \frac{r^2\big(\frac{x-x_0}{\eps}\big)r^2\big(\frac{y-x_0}{\eps}\big)}{|x-y|} \\
&=\frac{1}{2}\int |\nabla r|^2+\frac{1}{2}m|v_0|^2 
 +\int V(x_0+\eps x)r^2(x)-
\frac{1}{2}\iint \frac{r^2(x)r^2(y)}{|x-y|} \\
&={\mathcal E}(r)+\mathcal{H}(t)   
+\int V(x_0+\eps x)r^2(x) - m V(x_0).
\end{align*}
Taking into account that, since $\nabla^2 V$ is bounded
and 
\[
\int x \nabla V(x_0) r^2(x)=0,
\]
we have 
$$
\int V(x_0+\eps x)r^2(x)-mV(x_0)={\mathcal O}(\eps^2)
$$
and the assertion immediately follows.
\end{proof}

\noindent
Moreover we have

\begin{lemma}
\label{normestim} 
There exists $C>0$ such that
$\|\nabla u^\eps(t)\|_2\leq C\sqrt{\eps}$
for all $t\in [0,\infty)$ and $\eps>0$.
\end{lemma}
\begin{proof}
Taking into account that $V$ is bounded from below, that $E_\eps(0)$ is bounded with respect to $\eps$ by Lemma~\ref{espansione-energia}
and the energy and mass are conserved quantities, there exists a positive constant $C$ independent of $\eps$ such that,
for all $t\in [0,\infty)$ and $\eps>0$,
$$
\| \nabla u^\eps(t)\|_2^2
\leq \eps C+
\frac{1}{\eps^4}\iint \frac{|u^\eps(t,x)|^2|u^\eps(t,x)|^2}{|x-y|}.
$$
Now, by the Hardy-Littlewood and Gagliardo-Nirenberg inequalities yields
$$
\frac{1}{\eps^4}\iint \frac{|u^\eps(x,t)|^2|u^\eps(y,t)|^2}{|x-y|}
\leq \frac{C}{\eps^4}\|u^\eps(t)\|_{12/5}^4\leq 
C\sqrt{\eps}\left[\frac{\|u^\eps(t)\|_2^2}{\eps^3}\right]^{\frac{3}{2}}\|\nabla|u^\eps(t)|\|_2
\leq C\sqrt{\eps}\|\nabla u^\eps(t)\|_2.
$$
By combining the above inequalities 
the assertion follows.
\end{proof}

\noindent
First of all, let us define the momentum
\[
p^\eps (t,x):=\frac{1}{\eps^2} \mathfrak{Im} (\bar{u}^\eps(t,x) \nabla u^\eps(t,x)), \quad x\in{\mathbb R}^3,\,\,\, t\in [0,\infty).
\]

\begin{lemma}
\label{identities}
The following identities hold
\[
\begin{array}{lll}
\displaystyle\partial_t \frac{|u^\eps(t,x)|^2}{\eps^3}=-\operatorname{div}(p^\eps(t,x)), & & t\in [0,\infty),\; x\in{\mathbb R}^3, \\
& & \\
\displaystyle \partial_t \int  p^\eps (t,x)
= 
- \frac{1}{\eps^3}\int \nabla V(x) |u^\eps(t,x)|^2,
 & & t\in [0,\infty).\end{array}
\]
\end{lemma}
\begin{proof}
By multiplying the equation for $u^\eps$ by $\bar{u}^\eps$ and 
taking then the real part easily yields
the first identity, via trivial manipulations. 
Concerning the second identity, since 
$u^\eps\in C([0,\infty),H^2({\mathbb R}^3))
\cap C^1([0,\infty),L^2({\mathbb R}^3))$
by \cite[Theorem 5.2.1 and Remark 5.2.9]{cazenave} and
the map $t\mapsto \int p_j^\eps (t,x)$ 
($j=1,2,3$) is $C^1([0,\infty))$ (see e.g.\ \cite[Appendix A]{frolich1}), we have
\begin{align*}
\partial_t\int p_j^\eps (t,x)
=&
\frac{1}{\eps^2} 
\int\mathfrak{Im}(\partial_t \bar{u}^\eps \partial_j u^\eps)
- \frac{1}{\eps^2}
\int \mathfrak{Im}(\partial_j\bar{u}^\eps \partial_t u^\eps)
\\
=&
\frac{2}{\eps^2} 
\int\mathfrak{Im}(\partial_t \bar{u}^\eps \partial_j u^\eps)
\\
=&
-\frac{1}{\eps}
\int \operatorname{div} 
\left( \partial_j \mathfrak{Re}(u^\eps) 
\nabla \mathfrak{Re}(u^\eps) \right)
-\frac{1}{\eps} 
\int \operatorname{div} 
\left( \partial_j \mathfrak{Im}(u^\eps) 
\nabla \mathfrak{Im}(u^\eps) \right)
+ \frac{1}{2\eps} \int \partial_j (|\nabla u^\eps|^2)
\\
&
+\frac{1}{\eps^3} 
\int V(x)\partial_j |u^\eps|^2
- \frac{1}{\eps^5} 
\int\left(\frac{1}{|x|} * |\bar{u}^\eps|^2 \right) 
\partial_j|u^\eps|^2.
\end{align*}
The first three terms 
as well as the last one in the above identity integrate to zero. Furthermore, the integral involving the
nonlocal term is zero too, since it holds
\begin{align*}
\int\left(\frac{1}{|x|} * |\bar{u}^\eps|^2 \right) 
\partial_j|u^\eps|^2
&= 
-\int \partial_j \left(\frac{1}{|x|} * |\bar{u}^\eps|^2 \right) 
|u^\eps|^2
=
-\int  \left(\frac{1}{|x|} * \partial_j|\bar{u}^\eps|^2 \right)
|u^\eps|^2   \\
&=
-\int\left(\frac{1}{|x|} * |\bar{u}^\eps|^2 \right) 
\partial_j|u^\eps|^2.
\end{align*}
Hence, the assertion follows.
\end{proof}

\subsection{Proof of Theorem~\ref{thm:dyn} concluded}
Once Theorem~\ref{thm:enconv} holds true, the proof of Theorem~\ref{thm:dyn} 
proceeds as in \cite{Bro-Jerr,keraa}. Given $T_0>0$ to be choosen suitably small, the function
$$
\Psi^\eps(t,x):=e^{-\frac{i}{\eps}(\eps x+x(t))\cdot v(t)}u^\eps(t,\eps x+x(t))
$$
satisfies $\|\Psi^\eps(t,\cdot)\|_2=\|r\|_2$. Furthermore, taking into account 
Lemma~\ref{espansione-energia}, the conservation of ${\mathcal H}$ and the characterization
of $r$ as infimum on $\mathcal{M}$, by a direct computation we end up with
\begin{align*}
0\leq {\mathcal E}(\Psi^\eps(t))-{\mathcal E}(r) 
= & E_\eps (t) + \frac{1}{2} m|v(t)|^2-v(t)\int p^\eps(t,x)
-\frac{1}{\eps^3}\int V(x)|u^\eps(t,x)|^2-{\mathcal E}(r)\\
= & m|v(t)|^2-v(t)\int p^\eps(t,x)+m V(x(t))-\frac{1}{\eps^3}\int V(x)|u^\eps(t,x)|^2
+{\mathcal O}(\eps^2).
\end{align*}
In turn $0\leq {\mathcal E}(\Psi^\eps(t))-{\mathcal E}(r)\leq C\eta^\eps(t)+{\mathcal O}(\eps^2),$ where 
$\eta^\eps$ is defined in \cite[p.179]{keraa}
and satisfies $\eta^\eps(0)={\mathcal O}(\eps^2)$. By Theorem~\ref{thm:enconv}
we know that there exist $C,A>0$ such that
\[
\mathcal{E} (\phi) - \mathcal{E} (r) \geq C
\inf_{x\in\mathbb{R}^3,\,\theta\in [0,2\pi)}
\| \phi - e^{i\theta} r(\cdot - x) \|^2
\]
for any $\phi\in H^1(\mathbb{R}^3,\mathbb{C})$ such that 
$\| \phi \|_2 = \| r \|_2$, $\!\!\inf\limits_{x\in\mathbb{R}^3,\,\theta\in [0,2\pi)}
\| \phi - e^{i\theta} r(\cdot - x) \|\leq \|r\|$ and $\mathcal{E} (\phi) - \mathcal{E} (r)\leq A$.
Then, introducing
$$
T^\eps=\sup\Big\{t\in [0,T_0]: \eta^\eps(s)\leq A,\,\, \!\!\inf\limits_{x\in\mathbb{R}^3,\,\theta\in [0,2\pi)}
\| \Psi^\eps(s,\cdot) - e^{i\theta} r(\cdot - x) \|\leq \|r\|, \hbox{ for all } s\in [0,t]\Big\}
$$
and observing that $\Psi^\eps(0,x)=r(x),$ it follows that $T^\eps>0$ for any $\eps$ sufficiently small and there exist families
of functions $\theta^\eps:[0,2\pi)\to\mathbb{R}$ and $z^\eps:\mathbb{R}^3\to\mathbb{R}$ such that
$$
\Big\|u^\eps(t,x)- e^{\frac{i}{\eps}(x\cdot v(t)+\theta^\eps(t))}r\Big(\frac{x-z^\eps(t)}{\eps}\Big) \Big\|^2_{{\mathcal H}_\eps}
\leq C\eta^\eps(t)+{\mathcal O}(\eps^2),\quad\text{for all $t\in [0,T^\eps)$.}
$$
From this stage on, taking into account the mass and momentum identities of Lemma~\ref{identities}, 
the conclusion $\eta^\eps(t)\leq C\eps^2$ for all $t\in [0,T^\eps)$, and 
hence in turn for any $t\in [0,T_0]$, follows exactly as in \cite[Lemma 3.4-3.6]{Bro-Jerr}. The conclusion
of Lemma~\ref{normestim} is used in the proof of \cite[Lemma 3.5]{Bro-Jerr} to have $\|p^\eps(t)\|_1\leq C$
and choose in turn $T_0$ sufficiently small. Finally the assertion of Theorem~\ref{thm:dyn} follows by
mimicking the continuation argument exploited in \cite[p.185]{Bro-Jerr}.

\bigskip
\medskip


\bigskip
\medskip


\begin{thebibliography}{99}

\bibitem{ambr-malch}
{\sc A.\ Ambrosetti, A.\ Malchiodi}, 
Concentration phenomena for nonlinear Schr\"odinger equations: 
recent results and new perspectives. Perspectives in nonlinear partial differential equations, 
19--30, Contemp.\ Math., 446, Amer.\ Math.\ Soc., Providence, RI, 2007.

\bibitem{BGM1}
{\sc V.\ Benci, M.G.\ Ghimenti, A.M.\ Micheletti}, 
The nonlinear Schr\"odinger equation: soliton dynamics, 
{\em J. Differential Equations} {\bf 249} (2010), 3312--3341. 

\bibitem{BGM2}
{\sc V.\ Benci, M.G.\ Ghimenti, A.M.\ Micheletti}, 
On the dynamics of solitons in the nonlinear Schr\"oedinger equation, 
{\em Arch. Ration. Mech. Anal.} {\bf 205} (2012), 467--492.

\bibitem{BerSchub} 
{\sc F.A.\ Berezin, M.A.\ Shubin}, The Schr\"odinger equation. 
Translated from the 1983 Russian edition by Yu. Rajabov, D.A. 
Leites and N. A. Sakharova and revised by Shubin. With contributions by 
G.L. Litvinov and Leites. Mathematics and its 
Applications (Soviet Series), 66. Kluwer Academic Publishers Group, Dordrecht, 1991, 555pp.

\bibitem{Bro-Jerr}
{\sc J.C.\ Bronski, R.L.\ Jerrard}, 
Soliton dynamics in a potential, 
{\em Math. Res. Lett.} {\bf 7} (2–3) (2000), 329--342. 

\bibitem{cazenave} 
{\sc T. Cazenave}, 
An introduction to nonlinear Schr\"odinger equations,
Textos de M\'etodos Matem\'aticos
{\bf 26}, Universidade Federal do Rio de Janeiro 1996.

\bibitem{cazlio}
{\sc T.\ Cazenave, P.L.\ Lions}, 
Orbital stability of standing waves for some nonlinear Schrödinger equations, 
{\em Comm. Math. Phys.} {\bf 85} (1982), 549--561.


\bibitem{CSS} 
{\sc S.\ Cingolani, S.\ Secchi, M.\ Squassina},
Semi-classical limit for Schr\"odinger 
equations with magnetic field and Hartree-type nonlinearities,
{\em Proc.\ Roy.\ Soc.\ Edinburgh Sect.\ A} {\bf 140} (2010), 973--1009. 

\bibitem{DV}
{\sc K.\ Datchev, I.\ Ventura},
Solitary waves for the Hartree equation with a slowly varying 
potential,
{\em Pacific J.\ Math.} {\bf 248} (2010), 63--90.

\bibitem{frolich1}
{\sc J.\ Fr\"ohlich, S.\ Gustafson, B.L.G.\ Jonsson, I.M.\ Sigal}, 
Dynamics of solitary waves external potentials,  
{\em Comm.\ Math.\ Phys.}  {\bf 250} (2004), 613--642.

\bibitem{froltsaiyau}
{\sc J. Fr\"ohlich, Tai-Peng Tsai, Horng-Tzer Yau}, 
On the point-particle (Newtonian) limit of the non-linear Hartree equation,
{\em Comm.\ Math.\ Phys.} {\bf 225} (2002), 223--274.

\bibitem{gril1}
{\sc M. Grillakis, J. Shatah, W. Strauss},
Stability theory of solitary waves in the presence of symmetry. I,
{\em J.\ Funct.\ Anal.} {\bf 74} (1987), 160-197. 

\bibitem{gril2}
{\sc M. Grillakis, J. Shatah, W. Strauss},
Stability theory of solitary waves in the presence of symmetry. II,
{\em J.\ Funct.\ Anal.} {\bf 94} (1990), 308-348. 

\bibitem{keraa}
{\sc S.\ Keraani}, 
Semiclassical limit for nonlinear Schr\"odinger equation with potential II, 
{\em Asymptot. Anal.} {\bf 47} (2006), 171--186.

\bibitem{Lenz} 
{\sc E.\ Lenzmann}, 
Uniqueness of ground states for pseudorelativistic Hartree equations, 
{\em Anal.\ PDE} {\bf 2} (2009), 1--27.

\bibitem{Lieb} 
{\sc E.H.\ Lieb},
Existence and uniqueness of the minimizing solution of Choquard's nonlinear equation, 
{\em Stud.\ Appl.\ Math.} {\bf 57} (1977), 93--105. 

\bibitem{liebloss} 
{\sc E.H.\ Lieb, M. Loss}, Analysis. Second edition. 
Graduate Studies in Mathematics, 14. 
American Mathematical Society, Providence, RI, 2001, 346pp. 

\bibitem{Lions1} 
{\sc P.-L.\ Lions},  
The concentration-compactness principle in the 
calculus of variations. The locally compact case. I, 
{\em Ann. Inst. H. Poincar\'e Anal. Non Lin\'eaire} {\bf 1} (1984), 109--145. 

\bibitem{Lions2} 
{\sc P.-L.\ Lions},  
The concentration-compactness principle in the 
calculus of variations. The locally compact case. II, 
{\em Ann. Inst. H. Poincar\'e Anal. Non Lin\'eaire} {\bf 1} (1984), 223--283. 

\bibitem{MZ} 
{\sc L.\ Ma, L. Zhao}, 
Classification of positive solitary solutions of the nonlinear Choquard equation,
{\em Arch.\ Ration.\ Mech.\ Anal.} {\bf 195} (2010), 455--467. 

\bibitem{MPS} 
{\sc E.\ Montefusco, B.\ Pellacci, M.\ Squassina}, 
Energy convexity estimates for non-degenerate ground 
states of nonlinear 1D Schr\"odinger systems, 
{\em Commun.\ Pure Appl.\ Anal.} {\bf 9} (2010), 867--884.

\bibitem{morozschaft}
{\sc V.\ Moroz, J.\ Van Schaftingen}, 
Groundstates of nonlinear Choquard equations: existence, qualitative properties and decay asymptotics,
{\em preprint}.

\bibitem{pekar}
{\sc S.\ Pekar}, Untersuchung über die Elektronentheorie der Kristalle, Akademie Verlag, Berlin, 1954.

\bibitem{penrose}
{\sc R. Penrose}, 
Quantum computation, entanglement and state reduction, 
{\em Phil. Trans. R. Soc.} {\bf 356} (1998), 1--13.


\bibitem{tao-solitons}
{\sc T.\ Tao}, 
Why are solitons stable? 
{\em Bull.\ Amer.\ Math.\ Soc.} {\bf 46} (2009) 1-33. 

\bibitem{MT} 
{\sc P.\ Tod, I.\ Moroz}, 
An analytical approach to the Schr\"odinger-Newton equations,
{\em Nonlinearity} {\bf 12} (1999), 201--216. 

\bibitem{WW} 
{\sc J.\ Wei, M.\ Winter}, 
Strongly interacting bumps for 
the Schr\"odinger-Newton equations, 
{\em J. Math. Phys.} {\bf 50} (2009), 22 pp

\bibitem{weinstein1}
{\sc M.I.\ Weinstein}, 
Modulational stability of ground states of nonlinear Schrödinger equations, 
{\em SIAM J.\ Math.\ Anal.} {\bf 16} (1985) 472--491. 

\bibitem{weinstein2}
{\sc M.I.\ Weinstein}, 
Lyapunov stability of ground states of nonlinear dispersive evolution equations, 
{\em Comm.\ Pure Appl.\ Math.} {\bf 39} (1986), 51--67.


\end{thebibliography}
\end{document}